\documentclass[12pt]{amsart}

\usepackage{amsmath,amssymb,amsthm,enumerate}

\newtheorem{theorem}{Theorem}[section]

\newtheorem{lemma}[theorem]{Lemma}
\newtheorem{proposition}[theorem]{Proposition}
\newtheorem{conjecture}[theorem]{Conjecture}

\theoremstyle{definition}

\theoremstyle{remark}
\newtheorem{remark}[theorem]{Remark}

\newcommand{\bC}{\mathbb C}

\newcommand{\bS}{\mathbb S}
\newcommand{\bB}{\mathbb B}

\newcommand{\bN}{\mathbb N}

\def\b{\beta}
\def\a{\alpha}

\newcommand{\norm}{|\!|}

\DeclareMathOperator{\CFT}{CFT}

\begin{document}
\title[]{On the HJY Gap Conjecture in CR geometry vs. the SOS Conjecture for polynomials}
\author{Peter Ebenfelt}
\address{Department of Mathematics, University of California at San Diego, La Jolla, CA 92093-0112}
\email{pebenfel@math.ucsd.edu}
\thanks{The author was supported in part by the NSF grant DMS-1301282.}
\begin{abstract} We show that the Huang-Ji-Yin (HJY) Gap Conjecture concerning CR mappings between spheres follows from a conjecture regarding Sums of Squares (SOS) of polynomials. The connection between the two problems is made by the CR Gauss equation and the fact that the former conjecture follows from the latter follows from a recent result, due to the author, on partial rigidity of CR mappings of strictly pseudoconvex hypersurfaces into spheres.
\end{abstract}


\maketitle

\section{Introduction}

The purpose of this note is to explain how the Huang-Ji-Yin (HJY) Gap Conjecture concerning CR mappings between spheres \cite{HuangJiYin09} follows from a conjecture regarding Sums of Squares (SOS) of polynomials. The connection between the two problems is made by the CR Gauss equation (a well known fact) and the implication follows from a recent result, due to the author \cite{E12}, on partial rigidity ("flatness") of CR mappings of strictly pseudoconvex hypersurfaces into spheres.

The HJY Gap Conjecture concerns CR mappings $f$ of an open piece of the unit sphere $\bS^n\subset\bC^{n+1}$ into the unit sphere $\bS^N\subset \bC^{N+1}$ when the codimension $N-n$ lies in the integral interval $[0,D_n]$, where $D_n$ is a specific integer that depends on $n$ (with $D_n\sim \sqrt{2}n^{3/2}$, see below); here, we use the non-standard convention that the superscript $m$ on a real hypersurface $M^m\subset \bC^{m+1}$ refers to the CR dimension, and not the real dimension (which is $2m+1$). The mappings $f$ are assumed to be (sufficiently) smooth and, by results in \cite{Forstneric89} and \cite{CS90}, they therefore extend as rational maps without poles on $\overline{\bB_{n+1}}$, where $\bB_{n+1}\subset \bC^{n+1}$ denotes the unit ball. In particular, there is no loss of generality in considering globally defined CR mappings $f\colon \bS^{n}\to \bS^N$. The conjecture asserts that there is a collection of finitely many disjoint integral subintervals $I_1,\ldots, I_{\kappa_0}\subset [0,D_n]$ with the property that if the codimension $N-n$ belongs to one of these subintervals, $N-n\in I_\kappa=[a_\kappa,b_\kappa]$, then
\begin{equation}\label{f=TLf0}
f=T\circ L\circ f_0,
 \end{equation}
where $f_0$ is a CR mapping $S^n\to S^{N_0}$ for some $N_0$ with codimension $N_0-n<a_\kappa\leq N-n$ (in particular, then $N_0<N$), and where $L\colon S^{N_0}\to S^N$ is the standard linear embedding in which the last $N-N_0$ coordinates are zero and $T\colon S^N\to S^N$ is an automorphism of the target sphere $S^N$. It is well known and easy to see that the representation \eqref{f=TLf0} is equivalent to the statement that the image $f(\bS^n)$ is contained in an affine complex subspace $A^{N_0+1}$ of dimension $N_0+1$.

Before formulating the HJY Gap Conjecture more precisely, we must introduce the integral intervals $I_\kappa$. For $n\geq 2$, we define
\begin{equation}\label{Ik}
I_\kappa:=\left \{j\in \bN\colon (\kappa-1)n+\kappa\leq j\leq \sum_{i=0}^{\kappa-1}(n-i)-1=n+(n-1)+\ldots+(n-\kappa+1)-1\right\},
\end{equation}
for $\kappa=1,\ldots, \kappa_0$, where $\kappa_0=\kappa_0(n)$ is the largest integer $\kappa$ such that the integral interval $I_\kappa$ is non-trivial, i.e.,
\begin{equation}\label{k0}
(\kappa-1)n+\kappa\leq \sum_{i=0}^{\kappa-1}(n-i)-1.
\end{equation}
A simple calculation shows that $\kappa_0=\kappa_0(n)$ is increasing in $n$ (clearly, with $\kappa_0<n$) and grows like $\sqrt{2n}$. We have, e.g., $\kappa_0(2)=1$, $\kappa_0(4)=2$, and for $\kappa_0(n)\geq 3$, we need $n\geq 7$. For the integer $D_n$ referenced above, we can then take
$$
D_n=\kappa_0n-\frac{\kappa_0(\kappa_0-1)}{2}-1=\sqrt{2}n^{3/2}-n-\sqrt{2n}+O(1).
$$
Now, the conjecture made by X. Huang, S. Ji, and W. Yin in \cite{HuangJiYin09} can be formulated as follows:

\begin{conjecture}[HJY Gap Conjecture]\label{HJYConj} For $n\geq 2$, let $\kappa_0$ and $I_1,\ldots I_{\kappa_0}$ be as above and assume that $f\colon \bS^n\to \bS^{N}$ is a sufficiently smooth CR mapping. If the codimension $N-n\in I_\kappa$, then there exists an integer $n\leq N_0<N$ with
\begin{equation}\label{No-nest}
N_0-n\leq (\kappa-1)n-\kappa-1
 \end{equation}
and an affine complex subspace $A^{N_0+1}$ of dimension $N_0+1$ such that $f(\bS^n)\subset S^N\cap A^{N_0+1}$.
\end{conjecture}

The $\kappa$th integral interval $I_\kappa$ with the property described in the conjecture above is referred to as the $\kappa$th {\it gap}. We note that the existence of the first gap is the statement that if $f\colon \bS^n\to \bS^{N}$ is a sufficiently smooth CR mapping and $1\leq N-n\leq n-1$, then $f(\bS^n)\subset \bS^N\cap A^{n+1}$. Since $\bS^N\cap A^{n+1}$ is a sphere in the $(n+1)$-dimensional complex space $A^{n+1}$ and, thus, CR equivalent to $S^n\subset \bC^{n+1}$, we can write $f=T\circ L\circ f_0$, where $T$ and $L$ are as in \eqref{f=TLf0} and $f_0$ is a map of $S^n$ to itself. By work of Poincar\'e \cite{Poincare07}, Alexander \cite{Alexander74}, and Pinchuk \cite{Pincuk74}, $f_0$ is in fact an automorphism of $\bS^n$ (unless it is constant, of course) and by an appropriate choice of $T$, we can in fact make $f_0$ linear. The existence of the first gap, under the assumption that $f$ is real-analytic, was established by Faran in \cite{Faran86}; the smoothness required for this was subsequently lowered to $C^{N-n}$ by Forstneric \cite{Forstneric89} and then to $C^2$ by X. Huang in \cite{Huang99}. The existence of the second gap (when $n\geq 4$) and the third gap (when $n\geq 7$) was established under the assumption of $C^3$-smoothness of $f$ in \cite{HuangJiXu06} and \cite{HuangJiYin12}, respectively. The existence of the $\kappa$th gap for $3< \kappa\leq \kappa_0$ is an open problem at this time. It is, however, known \cite{JPDLebl09} that when the codimension $N-n$ is sufficiently large, then there are no more gaps (in the sense of Conjecture \ref{HJYConj}).

For the first three gaps, one can also classify the possible maps $f_0$ that appear in \eqref{f=TLf0}, as in the (very simple) Poincar\'e-Alexander-Pinchuk classification corresponding to the first gap described above; see \cite{HuangJi01}, \cite{Hamada05}, \cite{HuangJiYin12}. For the gaps beyond these, such a classification is most likely beyond what one can hope for at this time, at least for large $\kappa$. To the best of the author's knowledge, there is no conjecture as to what such "model" maps would be for general $\kappa$.

For a CR mapping $f\colon \bS^n\to \bS^N$, there is a notion of the CR second fundamental form of $f$ and its covariant derivatives, and if we form the corresponding sectional curvatures (defined more precisely in the next section), then we obtain a collection of polynomials $\Omega^1(z), \ldots,\Omega^{N-n}(z)$ in the variables $z=(z^1,\ldots,z^n)\in\bC^n$, whose coefficients consist of components of the second fundamental form and its covariant derivatives up to some finite order (bounded from above by the codimension $N-n$); we shall refer to the polynomial mapping $\Omega=(\Omega^1,\ldots\Omega^{N-n})$ as the total second fundamental polynomial. These polynomials satisfy a Sums Of Squares (SOS) identity as a consequence of a CR version of the Gauss equation. The SOS identity has the following form
\begin{equation}\label{SOSid}
\sum_{j=1}^{N-n}|\Omega^j(z)|^2 = A(z,\bar z)\sum_{i=1}^n|z^i|^2,
\end{equation}
where $A(z,\bar z)$ is a Hermitian (real-valued) polynomial in $z$ and $\bar z$. To simplify the notation, for a polynomial mapping $P(z)=(P^1(z),\ldots,P^q(z))$ we shall write $\norm P(z)\norm^2$ for the SOS of moduli of the components, i.e.,
\begin{equation}
\norm P(z)\norm^2:=\sum_{k=1}^q|P^k(z)|^2.
\end{equation}
The number $q$ of terms in the norm will differ depending on the mapping in question, but will be clear from the context. Using this notation, the identity \eqref{SOSid} can be written in the following way:
\begin{equation}\label{SOSid'}
\norm\Omega(z)\norm^2=A(z,\bar z)\norm z\norm^2.
\end{equation}
The polynomial $A(z,\bar z)$ is in principle computable from $f$, but useful properties of $A$ seem difficult to extract directly in this way, and often it suffices to know that $\Omega$ satisfies an SOS identity of this form, for some Hermitian polynomial $A$. SOS identities of the form \eqref{SOSid'} appear in many different contexts, and there is an abundance of literature considering various aspects of such identities. We mention here only a few, and only ones with a connection to CR geometry and complex analysis: e.g., \cite{Quillen68}, \cite{CatlinDangelo96}, \cite{CatlinJPD99}, \cite{EHZ05}, \cite{JPDLebl09}, \cite{JPD11}, \cite{HuangY13}, \cite{GrHa13}, \cite{GruLV14}, \cite{E15}, and refer the reader to these papers for further connections and references to the literature. The reader is especially referred to the paper  \cite{JPD11} by D'Angelo, which contains an excellent discussion of SOS identities and positivity conditions.

We shall here be concerned with a very specific property of polynomial maps $\Omega$ that satisfy \eqref{SOSid'}, namely the possible linear ranks that can occur. For a polynomial mapping $P(z)=(P^1(z),\ldots,P^q(z))$, we define its {\it linear rank} to be the dimension of the complex vector space $V_P$ spanned by its components, in the polynomial ring $\bC[z]$. The main result in this note is that the HJY Gap Conjecture will follow from the following conjecture regarding the possible linear ranks of polynomial mappings $P(z)$ that satisfy an SOS identity:

\begin{conjecture}[SOS Conjecture]\label{SOSConj} Let $P(z)=(P^1(z),\ldots,P^q(z))$ be a polynomial mapping in $z=(z^1,\ldots, z^n)\in \bC^n$, and assume that there exists a Hermitian polynomial $A(z,\bar z)$ such that the SOS identity
\begin{equation}\label{SOSidP}
\norm P(z)\norm^2=A(z,\bar z)\norm z\norm^2
\end{equation}
holds. If $r$ denotes the linear rank of $P(z)$, then either
\begin{equation}\label{rmax}
r\geq (\kappa_0+1)n-\frac{\kappa_0(\kappa_0+1)}{2}-1,
\end{equation}
where $\kappa_0$ is the largest integer $\kappa$ such that \eqref{k0} holds, or there exists a integer $1\leq \kappa\leq \kappa_0<n$ such that
\begin{equation}\label{e:ASOSbound}
\sum_{i=0}^{\kappa-1} (n-i)=n\kappa-\frac{\kappa(\kappa-1)}{2}\leq r\leq \kappa n.
\end{equation}
\end{conjecture}

\begin{remark} {\rm The integer $\kappa_0$ is also the integer for which the integral intervals in $\kappa$, defined by \eqref{e:ASOSbound} start overlapping for $\kappa=\kappa_0+1$.
}
\end{remark}

The main result in this note is that this SOS Conjecture implies the HJY Gap Conjecture:

\begin{theorem}\label{MainThm} If the SOS Conjecture $\ref{SOSConj}$ holds, then the HJY Gap Conjecture $\ref{HJYConj}$ holds.
\end{theorem}

The connection between the two conjectures is explained in Section \ref{GaussSec}. The conclusion of Theorem \ref{MainThm} will then be derived, in Section \ref{Proof}, as a consequence of Theorem 1.1 in \cite{E12}, reproduced here in a special case as Theorem \ref{Thm1.1}.

\subsection{Results on the SOS Conjecture; reduction to an alternative SOS Conjecture} While the literature on SOS of polynomials is vast, as mentioned above, there are very few results that have a direct impact on the SOS Conjecture \ref{SOSConj}. To the best of the author's knowledge, the only general result on this conjecture is what is now known as Huang's Lemma, which first appeared in \cite{Huang99}, and which establishes the first gap in the SOS Conjecture: If $r<n$, then $A\equiv 0$, and, hence $r=0$. Huang used this result in \cite{Huang99} to give a new proof of  Faran's result regarding existence of the first gap in the Gap Conjecture \ref{HJYConj}, and to show that it suffices to assume that the mappings are merely $C^2$-smooth.

In another recent paper \cite{GrHa13} by Grundmeier and Halfpap, the SOS Conjecture \ref{SOSConj} was established in the special case where $A(z,\bar z)$ is itself an SOS, i.e.,
\begin{equation}\label{ASOS}
A(z,\bar z)=\norm F(z)\norm^2,
\end{equation}
for some polynomial mapping $F(z)$. The integer $\kappa$ in the conjecture in this case is the linear rank of the polynomial mapping $F(z)$; it is assumed in \cite{GrHa13} that the components of $P(z)$ are homogeneous polynomials, but a simple homogenization argument can remove this assumption (cf. \cite{E15}). The Grundmeier-Halfpap result by itself does not seem to have any direct implications for the Gap Conjecture \ref{HJYConj}, as the needed information regarding the Hermitian polynomial $A(z,\bar z)$ seems difficult to glean from the mapping $f$, but it offers the opportunity to formulate an alternative, arguably simplified version of the SOS conjecture, which would imply Conjecture \ref{SOSConj} as a consequence of the Grundmeier-Halfpap result. We shall formulate this alternative SOS Conjecture in what follows.

We observe that, by standard linear algebra arguments, any Hermitian polynomial $A(z,\bar z)$ can be expressed as a difference of squared norms of polynomial mappings,
\begin{equation}\label{DOS}
A(z,\bar z)=\norm F(z)\norm^2-\norm G(z)\norm^2,
\end{equation}
where $F=(F^1,\ldots,F^{q_+})$ and $G=(G^1,\ldots, G^{q_-})$ are mappings whose components are polynomials in $z$. We may further assume that the complex vector spaces $V_F$, $V_G$ spanned by their respective components have dimensions $q_+$, $q_-$, respectively (i.e., the components of $F$ and $G$ are linearly independent, so their linear ranks are $q_+$, $q_-$, respectively), and that $V_F\cap V_G=\{0\}$. The Grundmeier-Halfpap result proves Conjecture \ref{SOSConj} in the special case where $G=0$. Thus, it suffices to prove the conjecture in the case where $G\neq 0$. In this case, the product $A(z,\bar z)\norm z\norm^2$ need of course not be an SOS, so this must be assumed. An optimistic view of the situation in the conjecture would be to hope that the "gaps" in linear ranks that are predicted in \eqref{e:ASOSbound} can only occur when $G=0$, and when $G\neq 0$, but $A(z,\bar z)\norm z\norm^2$ is still an SOS, the lower bound \eqref{rmax} always holds. The author has reasons to believe that this optimistic view is indeed what happens, though at this point the reasons are too vague to try to explain in this note. In any case, the following "weak", or alternative form of the SOS Conjecture, if true, then implies the SOS Conjecture \ref{SOSConj}, in view of the Grundmeier-Halfpap result.

\begin{conjecture}[Weak (Alternative) SOS Conjecture]\label{SOSConj2} Let $P(z)=(P^1(z),\ldots,P^q(z))$ be a polynomial mapping in $z=(z^1,\ldots, z^n)\in \bC^n$, and assume that there exists a Hermitian polynomial $A(z,\bar z)$ of the form \eqref{DOS} such that the SOS identity \eqref{SOSid} holds. If $r$ denotes the linear rank of $P(z)$ and if the polynomial mapping $G$ in \eqref{DOS} is not identically zero, then \eqref{rmax} holds.
\end{conjecture}

One of the main difficulties in Conjecture \ref{SOSConj2} when $G\neq 0$ comes from the fact that it seems hard to characterize when $A(z,\bar z)\norm z\norm^2$ is in fact an SOS of the form \eqref{SOSid}. The reader is referred to, e.g., \cite{JPDVar04}, \cite{JPD11} for discussions related to this difficulty. We can mention here that a necessary condition for an SOS identity \eqref{e:ASOSbound} to hold is that $V_{G\otimes z}\subset V_{F\otimes z}$, where the tensor product of two mappings $F\otimes H$ is defined as the mapping whose components comprise all the products of components $F^jH^k$. From this one can easily see that the linear rank $r=\dim_\bC V_P$ in Conjecture \ref{SOSConj2} must satisfy
\begin{equation}\label{specrk}
\dim_\bC V_{F\otimes z}/V_{G\otimes z}\leq r\leq \dim_\bC V_{F\otimes z}.
\end{equation}
The lower bound can only be realized if a maximum number of "cancellations" occur. If we consider the 1-parameter family of Hermitian polynomials
$$
A_t(z,\bar z):=\norm F(z)\norm^2-t\norm G(z)\norm^2
$$
for $0\leq t\leq 1$, where $A(z,\bar z)=A_1(z,\bar z)$ satisfies an SOS identity \eqref{e:ASOSbound}, then clearly $A_t(z,\bar z)\norm z\norm^2$ is an SOS for each $0\leq t\leq 1$ (since $A_t(z,\bar z)=A_1(z,\bar z)+(1-t)\norm G(z)\norm^2$). One can show that "cancellations" causing strict inequality in the upper bound in \eqref{specrk} do not occur for general $t$ in this range, and the linear rank of $A_t(z,\bar z)\norm z\norm^2$ for such $t$ is then $r=\dim_\bC V_{F\otimes z}$. Nevertheless, for the given $A(z,\bar z)=A_1(z,\bar z)$, all we can say seems to be that the estimate \eqref{specrk} holds.

\section{The second fundamental form and the Gauss equation}\label{GaussSec}

We shall utilize E. Cartan's differential systems ("moving frames") approach to CR geometry, as well as S. Webster's theory of psuedohermitian structures. We will follow the set-up and notational conventions introduced in \cite{BEH08} (see also \cite{E15} and \cite{EHZ04}). We shall summarize the notation very briefly here, but refer the reader to \cite{BEH08} (which, on occasion, refers to \cite{EHZ04}) for all details. We shall also from the beginning specialize the general set-up to the special case of CR mappings between spheres, which simplifies matters significantly due to the vanishing of the CR curvature tensor of the sphere. Thus, let $f\colon \bS^n\to \bS^N$ be a smooth CR mapping with $2\leq n\leq N$. For a point $p_0\in \bS^n$, we may choose local adapted (to $f$), admissible (in the sense of Webster \cite{Webster78}) CR coframes $(\theta,\theta^\alpha,\theta^{\bar\alpha})$ on $\bS^n$ near $p_0$ and $(\hat \theta,\hat \theta^A,\hat \theta^{\bar A})$ on $\bS^N$ near $\hat p_0:=f(p_0)$, where the convention in \cite{BEH08} dictates that Greek indices, $\alpha$, etc., range over $\{1,\ldots, n\}$,  capital Latin letters, $A$, etc., range over $\{1,\ldots N\}$, and where barring an index on a
 previously defined object corresponds to complex conjugation, e.g., $\theta^{\bar\alpha}:=\overline{\theta^\alpha}$. Being adapted means that
\begin{equation}
f^*\hat\theta=\theta,\quad
f^*\hat\theta^\alpha=\theta^\alpha,\quad f^*\hat\theta^a=0,
\end{equation}
where we have used the further convention that lower case Latin letters $a$, etc., run over the indices $\{N-n+1,\ldots, N\}$. Thus, in particular, $f$ is a (local) pseudohermitian mapping between the (local) pseudohermitian structures obtained on $\bS^n$ and $\bS^N$ by fixing the contact forms $\theta$ and $\hat \theta$ near $p_0$ and $\hat p_0$, respectively. We denote by $g_{\alpha\bar\beta}$, $\hat g_{A\bar B}$ the respective Levi forms (which can, and later will be both assumed to be the identity), and by $\omega_\alpha{}^\beta$, $\hat \omega_A{}^B$ the Tanaka-Webster connection forms. We shall pull all forms and tensors back to $\bS^n$ by $f$, and for convenience of notation, we shall simply denote by $\hat\omega_A{}^B$ the pulled back form $f^*\hat\omega_A{}^B$, etc. Moreover, the fact that the two coframes are adapted implies that we can drop the $\hat{}$ on the pullbacks to $\bS^n$ without any risk of confusion; in other words, we have, e.g., $\omega_\alpha{}^\beta=\hat\omega_\alpha{}^\beta$ and $g_{\alpha\bar\beta}=\hat g_{\alpha\bar\beta}$ (we repeat here that we refer to \cite{BEH08} and \cite{EHZ04} for the details), and of course $\omega_\alpha{}^a$, e.g., can have only one meaning.

The collection of 1-forms $(\omega_{\alpha}^{\:\:\:a})$ on $\bS^n$ defines the \emph{second fundamental form} of the mapping $f$, denoted $\Pi_f\colon T^{1,0}\bS^n\times T^{1,0}\bS^n\to T^{1,0}\bS^N/f_*T^{1,0}\bS^n$, as described in \cite{BEH08}. We recall from there that
\begin{equation}\label{SFF1}
\omega_{\alpha}^{\:\:\:a} = \omega_{\alpha \:\:\: \beta}^{\:\:\:a}\theta^{\beta}, \qquad \omega_{\alpha \:\:\: \beta}^{\:\:\:a} = \omega_{\beta \:\:\: \alpha}^{\:\:\:a}.
\end{equation}
If we identify the CR-normal space $T_{f(p)}^{1,0}\bS^N/f_*T_{p}^{1,0}\bS^n$, also denoted by $N_{p}^{1,0}{\bS^n}$, with $\mathbb{C}^{N-n}$, then we may identify $\Pi_f$ with the $\bC^{N-n}$-valued, symmetric $n\times n$ matrix $(\omega_{\alpha}{}^a{}_ \beta)_{a=n+1}^{N}$. We shall not be so concerned with the matrix structure of this object, and consider $\Pi_f$ as the collection, indexed by $\alpha, \beta$, of its component vectors $(\omega_{\alpha}{}^a{}_ \beta)_{a=n+1}^{N}$ in $\mathbb{C}^{N-n}$.  By  viewing the second fundamental form as a section over $\bS^n$ of the bundle
$(T^*)^{1,0}\bS^n\otimes N^{1,0}{\bS^n} \otimes (T^*)^{1,0}\bS^n$, we may use the pseudohermitian connections on $\bS^n$ and $\bS^N$ to define the covariant differential
\begin{equation*}
\nabla \omega_{\alpha\:\:\beta}^{\:\:a} = d\omega_{\alpha\:\:\beta}^{\:\:a} - \omega_{\mu\:\:\beta}^{\:\:a}\omega_{\alpha}^{\:\:\mu} + \omega_{\alpha\:\:\beta}^{\:\:b}\omega_{b}^{\:\:a} - \omega_{\alpha\:\:\mu}^{\:\:a}\omega_{\beta}^{\:\:\mu}.
\end{equation*}
We write $\omega_{\alpha\:\:\beta ; \gamma}^{\:\:a}$ to denote the component in the direction $\theta^{\gamma}$ and define higher order derivatives inductively as:
\begin{equation*}
\nabla \omega_{\gamma_{1}\:\:\gamma_{2};\gamma_{3}\ldots\gamma_{j}}^{\:\:a} = d\omega_{\gamma_{1}\:\:\gamma_{2};\gamma_{3}\ldots\gamma_{j}}^{\:\:a} + \omega_{\gamma_{1}\:\:\gamma_{2};\gamma_{3}\ldots\gamma_{j}}^{\:\:b}\omega_{b}^{\:\:a} - \sum_{l=1}^{j}\omega_{\gamma_{1}\:\:\gamma_{2};\gamma_{3}\ldots\gamma_{l-1}\mu
\gamma_{l+1}\ldots\gamma_{j}}^{\:\:a}\omega_{\gamma_{l}}^{\:\:\mu}.
\end{equation*}
A tensor
$T_{\a_1\ldots\a_r\bar\b_1\ldots\bar\b_s}{}^{a_1\ldots a_t\bar
b_1\ldots\bar b_q}$, with $r,s\geq1$, is called  {\em conformally flat} if it is a
linear combination of $g_{\a_i\bar\b_j}$ for $i=1,\ldots,r$,
$j=1,\ldots,s$, i.e.
\begin{equation}
T_{\a_1\ldots\a_r\bar\b_1\ldots\bar\b_s}{}^{a_1\ldots a_t\bar
b_1\ldots\bar b_q}=\sum_{i=1}^r\sum_{j=1}^s g_{\alpha_i\bar\beta_j}
(T_{ij})_{\a_1\ldots\widehat{\a_i}\ldots
\a_r\bar\b_1\ldots\widehat{\bar\b_j}\ldots\ldots\bar\b_s}{}^{a_1\ldots
a_t\bar b_1\ldots\bar b_q},
\end{equation}
where e.g.\ $\widehat{\a}$ means omission of that factor. (A similar definition can be made for tensors with different
orderings of indices.) The following observation gives a motivation for this definition. Let $T_{\a_1\ldots\a_r\bar\b_1\ldots\bar\b_s}{}^{a_1\ldots a_t\bar b_1\ldots\bar b_q}$ be a tensor, symmetric in $\alpha_1,\ldots,\alpha_r$ as well as in $\beta_1,\ldots,\beta_s$, and form the homogeneous vector-valued polynomial of bi-degree $(r,s)$ whose components are given by $$T^{a_1\ldots a_t\bar b_1\ldots\bar b_q}(z,\bar z):=
T_{\a_1\ldots\a_r\bar\b_1\ldots\bar\b_s}{}^{a_1\ldots a_t\bar b_1\ldots\bar b_q}z^{\alpha_1}\ldots z^{\alpha_r}\overline{z^{\beta_1}}\ldots\overline{z^{\beta_s}},
$$
where $z=(z^1,\ldots,z^n)$ and the usual summation convention is used. Then, the reader can check that the tensor is conformally flat if and only if all the polynomials $T^{a_1\ldots a_t\bar b_1\ldots\bar b_q}(z,\bar z)$ are divisible by the Hermitian form $g(z,\bar z):=g_{\alpha\bar\beta}z^\alpha\overline{z^\beta}$. Moreover, and importantly, a conformally flat tensor has the property that its covariant derivatives are again conformally flat, since one of the defining properties of the pseudohermitian connection is that $\nabla g_{\alpha\bar\beta}=0$. We shall use the terminology that $T_{\a_1\ldots\a_r\bar\b_1\ldots\bar\b_s}{}^{a_1\ldots a_t\bar b_1\ldots\bar b_q}\equiv 0 \mod \CFT$ if the tensor is conformally flat.

Now, the Gauss equation for the second fundamental form of a CR mapping $f\colon \bS^n\to \bS^N$ takes the following simple form (since the CR curvature tensors of $\bS^n$ and $\bS^N$ vanish):
\begin{equation}\label{Gauss}
g_{a\bar b}\omega_{\alpha}{}^a{}_{\nu}\omega_{\bar\beta}{}^{\bar b}{}_{\bar\mu}\equiv 0\mod\CFT.
\end{equation}
We proceed as in the proof of Theorem 5.1 in \cite{BEH08} and take repeated covariant derivatives in $\theta^{\gamma_r}$ and $\theta^{\bar \lambda_s}$ in the Gauss equation. By using the fact that $\omega_\alpha{}^a{}_{\beta;\bar\mu}$ is conformally flat (Lemma 4.1 in \cite{BEH08}) and the  commutation formula in Lemma 4.2 in \cite{BEH08}, we obtain the full family of Gauss equations, for any $r,s\geq 2$:
\begin{equation}\label{GaussFull}
g_{a\bar b}\omega_{\gamma_1}{}^a{}_{\gamma_2;\ldots\gamma_r}\omega_{\bar\lambda_1}{}^{\bar b}{}_{\bar\lambda_2;\ldots\bar\lambda_s}\equiv 0\mod\CFT.
\end{equation}
We now consider also the component vectors of higher order derivatives of $\Pi_f$ as elements of $\mathbb{C}^{N-n}\cong N_p^{1,0}S^n$ and define an increasing sequence of vector spaces
\begin{equation*}
E_{2}(p) \subseteq \ldots \subseteq E_{l}(p) \subseteq \ldots \subseteq \mathbb{C}^{N-n}\cong N_p^{1,0}\bS^n
\end{equation*}
by letting $E_{l}(p)$ be the span of the vectors
\begin{equation}\label{Eldef}
(\omega_{\gamma_{1}\:\:\gamma_{2};\gamma_{3}\ldots\gamma_{j}}^{\:\:a})_{a=n+1}^{N}, \qquad \forall\, 2 \leq j \leq l, \gamma_{j}\in \{1,\ldots,n\},
\end{equation}
evaluated at $p \in \bS^n$. We let $d_l(p)$ be the dimension of $E_l(p)$, and for convenience we set $d_1(p)=0$. As is mentioned in \cite{E12}, it is shown in \cite{EHZ04} that $d_l(p)$ defined in this way coincides with the $d_l(p)$ defined by (1.3) in \cite{E12}. By moving to a nearby point $p_0$ if necessary, we may assume that all $d_l=d_l(p)$ are locally constant near $p_0$ and
\begin{equation}\label{dimstab}
0=d_1<d_2<\ldots<d_{l_0}=d_{l_0+1}=\ldots\leq N-n
\end{equation}
for some $1\leq l_0\leq N-n+1$ (with $l_0=1$ if $d_2=0$ near such generic $p_0$). The mapping $f$ is said to be constantly $l_0$-degenerate of rank $d:=d_{l_0}\leq N-n$ at $p_0$; the codimension $N-n-d$ is called the degeneracy and if the degeneracy is $0$, then the mapping is also said to be $l_0$-nondegenerate.

For each integer $l\geq 2$, we form the $\bC^{N-n}$-valued, homogeneous polynomial $\Omega_{(l)}=(\Omega^1_{(l)},\ldots,\Omega^{N-n}_{(l)})$ in $z=(z^1,\ldots,z^n)\in \bC^n$ as follows:
\begin{equation}\label{Omega(l)}
\Omega^j_{(l)}(z):= \omega_{\gamma_{1}\:\:\gamma_{2};\gamma_{3}\ldots\gamma_{l}}^{\:\:a}z^{\gamma_1}\ldots z^{\gamma_l}, \quad a=n+j,
\end{equation}
and we define the {\it total second fundamental polynomial} $\Omega=(\Omega^1,\ldots,\Omega^{N-n})$ of $f$ near $p_0$ as follows:
\begin{equation}
\Omega^j(z):=\sum_{l=2}^{l_0}\Omega^j_{(l)}(z),
\end{equation}
where $l_0$ is the integer, defined above, where the dimensions $d_l$ stabilize. The following proposition is easily proved by using the fact that the rank of a matrix equals that of its transpose; the details are left to the reader.

\begin{proposition}\label{Omegarank} The rank $d=d_{l_0}$ of the $l_0$-degeneracy is also the linear rank of the polynomial mapping $\Omega(z)$, i.e., the dimension of the vector space in $\bC[z]$ spanned by the polynomials $\Omega^1(z),\ldots, \Omega^{N-n}(z)$.
\end{proposition}

We now recall, as mentioned above, that we may choose the adapted, admissible CR coframes (near $p_0$ and $\hat p_0=f(p_0)$) in such a way that the Levi forms of $\bS^n$ and $\bS^N$ both equal the identity matrix. Let us now insist on such a choice of coframes. We then notice that the full family of Gauss equations in \eqref{GaussFull} for $r,s\leq l_0$ can be summarized in the following Sum-Of-Squares identity for the total second fundamental polynomial.

\begin{lemma}[Total polynomial Gauss equation]\label{GaussLemma} There exists a Hermitian polynomial $A(z,\bar z)$ such that
\begin{equation}\label{GaussLemmaId}
\norm\Omega(z)\norm^2=A(z,\bar z)\norm z\norm^2,
\end{equation}
where the notation $\norm\Omega(z)\norm^2:=\sum_{j=1}^{N-n} |\Omega^j(z)|^2$ introduced in the introduction has been used.
\end{lemma}

\begin{proof} The proof consists of multiplying the identities \eqref{GaussFull} by $z^{\gamma_1}\ldots z^{\gamma_r}\overline{z^{\lambda_1}\ldots  z^{\lambda_s}}$ and summing according to the summation convention. The conformally flat tensors on the right hand sides all contain a factor of $\norm z\norm^2$. The proof is then completed by comparing the polynomial identities obtained in this way to the result of expanding the left hand side of \eqref{GaussLemmaId} and collecting terms of a fixed bidegree $(r,s)$. The details are left to the reader.
\end{proof}

\section{Proof of Theorem \ref{MainThm}}\label{Proof}

We shall prove Conjecture \ref{HJYConj} under the assumption that the conclusion of Conjecture \ref{SOSConj} holds. We quote first Theorem 1.1 in \cite{E12}, in the special case of CR mappings $f\colon \bS^n\to \bS^N$:

\begin{theorem}[\cite{E12}]\label{Thm1.1} Let $f\colon \bS^n \to  \bS^N$ be a smooth CR mapping and the dimensions $d_l(p)$ be as defined in Section $\ref{GaussSec}$. Let $U$ be an open subset of $\bS^n$ on which $f$ is constantly $l_0$-degenerate, and such that $d_l=d_l(p)$, for $2\leq l\leq l_0$, are constant on $U$ and \eqref{dimstab} holds. Assume that there are integers $0\leq k_2, k_3,\ldots,k_{l_0}\leq n-1$, such that:
\begin{equation}
\begin{aligned}\label{conds0}
d_l-d_{l-1} <&\sum_{j=0}^{k_l}(n-j),\quad l=2,\ldots, l_0,\quad (d_1=0)\\
k:=&\sum_{l=2}^{l_0}k_l<n.
\end{aligned}
\end{equation}
Then $f(\bS^n)$ is contained in a complex affine subspace $A^{n+d+k+1}$ of dimension $n+d+k+1$, where $k$ is defined in \eqref{conds0} and $d:=d_{l_0}$ is the rank of the $l_0$-degeneracy.
\end{theorem}

\begin{remark} {\rm The integers $k_2,\ldots, k_{l_0}$ become invariants of the mapping $f$ if we require them to be minimal in an obvious way. The invariant $k_2$ was introduced in \cite{Huang03} and called there the geometric rank of $f$. This geometric rank plays an important role in \cite{Huang03}, \cite{HuangJiXu06}, and \cite{HuangJiYin12}.
}
\end{remark}

\begin{proof}[Proof of Theorem $\ref{MainThm}$] We assume now that there is a mapping $f\colon \bS^n \to  \bS^N$ with codimension $N-n\in I_\kappa$ for some $\kappa\leq \kappa_0<n$. Thus, we have
$$
N-n\leq \sum_{i=0}^{\kappa-1}(n-i)-1.
$$
We consider an open subset $U\subset \bS^n$ as in Theorem \ref{Thm1.1}.
Since the rank of the $l_0$-degeneracy satisfies $d\leq N-n$, we then have
\begin{equation}\label{dest}
d\leq \sum_{i=0}^{\kappa-1}(n-i)-1,
\end{equation}
in Theorem \ref{Thm1.1}. By Proposition \ref{Omegarank}, $d$ is also the linear rank of the total second fundamental polynomial $\Omega(z)$, and by Lemma \ref{GaussLemma}, an SOS identity of the form \eqref{GaussLemmaId} holds. If we now assume that the SOS Conjecture \ref{SOSConj} holds, then \eqref{dest} implies that in fact
\begin{equation}\label{SOScons}
d\leq (\kappa-1)n.
\end{equation}
It is also clear from \eqref{dest} that there exist integers $0\leq k_l\leq \kappa-1$ such that the first identity in \eqref{conds0} hold. We shall choose the $k_j$ minimal, so that in addition we have
\begin{equation}\label{klmin}
d_l-d_{l-1}\geq \sum_{j=0}^{k_l-1}(n-j),
\end{equation}
where the right hand side is understood to be $0$ if $k_l=0$. We claim that
\begin{equation}\label{kclaim}
k:=\sum_{l=2}^{l_0}k_l \leq \kappa-1.
\end{equation}
If we can prove this claim, then it follows from Theorem \ref{Thm1.1}, since $\kappa\leq \kappa_0<n$, that $f(\bS^n)$ is contained in a complex affine subspace $A^{N_0+1}$ of dimension $N_0=n+d+k$, and the codimension satisfies, by \eqref{SOScons} and \eqref{kclaim},
$$
N_0-n=d+k\leq (\kappa-1)n+\kappa-1,
$$
which is precisely the desired conclusion in the Gap Conjecture \ref{HJYConj}. Thus, we proceed to prove \eqref{kclaim}. Let us denote by $g(j)$ the non-increasing function
\begin{equation}
g(j)=\left\{
\begin{aligned}
n-j,&\quad 0\leq j< n\\
0,&\quad j\geq n.
\end{aligned}
\right.
\end{equation}
Using the fact that we have set $d_1=0$, we can telescope $d$ as follows
\begin{equation}
d=(d_{l_0}-d_{l_0-1})+\ldots (d_2-d_1)=\sum_{l=2}^{l_0}(d_l-d_{l-1}),
\end{equation}
and deduce from \eqref{klmin} that
\begin{equation}\label{dest2}
d\geq \sum_{l=2}^{l_0}\sum_{j=0}^{k_l-1}(n-j)=\sum_{l=2}^{l_0}\sum_{j=0}^{k_l-1}g(j).
\end{equation}
Since $g(j)$ is non-increasing, we can estimate
\begin{equation}\label{shift}
\sum_{l=2}^{l_0}\sum_{j=0}^{k_l-1}g(j)\geq
\sum_{l=2}^{l_0}\sum_{j=0}^{k_l-1}g\left(j+m_l\right),
\end{equation}
where we have set $m_2=0$ and, for $3\leq l\leq l_0$,
\begin{equation}
m_l:=\sum_{i=2}^{l-1}k_{i}.
\end{equation}
Substituting $i=j+m_l$ in \eqref{shift}, we deduce from \eqref{dest2}
\begin{equation}
d\geq \sum_{l=2}^{l_0}\sum_{i=m_l}^{m_l+k_l-1}g(i)=\sum_{l=2}^{l_0}\sum_{i=m_l}^{m_{l+1}-1}g(i)=
\sum_{i=0}^{m_{l_0+1}-1}g(i).
\end{equation}
Since $m_{l_0+1}=k$, we conclude that
\begin{equation}
d\geq
\sum_{i=0}^{k-1}g(i),
\end{equation}
and since $k<n$, we also have $g(i)=n-i$ for $i=1,\ldots, k-1$, and therefore we can write
\begin{equation}
\sum_{i=0}^{k-1}(n-i)\leq d.
\end{equation}
By comparing this with \eqref{dest}, we conclude that $k-1<\kappa-1$, which establishes the claim \eqref{kclaim}. This completes the proof of Theorem \ref{MainThm}.

\end{proof}


\def\cprime{$'$}

\end{document}